\documentclass{amsart}

\newcommand{\w}{\omega}
\newtheorem{theorem}{Theorem}

\newtheorem{corollary}{Corollary}
\newtheorem{example}{Example}

\title[Strongly $\sigma$-metrizable spaces are super $\sigma$-metrizable]{Strongly $\sigma$-metrizable spaces are\\ super $\sigma$-metrizable}
\author{Taras Banakh}
\address{Ivan Franko National University of Lviv, Ukraine}
\email{t.o.banakh@gmail.com}

\begin{document}
\begin{abstract}
A topological space $X$ is called {\em strongly $\sigma$-metrizable} if $X=\bigcup_{n\in\w}X_n$ for an increasing sequence $(X_n)_{n\in\w}$ of closed metrizable subspaces such that every convergence sequence in $X$ is contained in some $X_n$. If, in addition, every compact subset of $X$ is contained in some $X_n$, $n\in\w$, then $X$ is called {\em super $\sigma$-metrizable}. Answering a question of V.K.~Maslyuchenko and O.I.~Filipchuk, we prove that a topological space is strongly $\sigma$-metrizable if and only if it is super $\sigma$-metrizable.
\end{abstract}
\maketitle

Following \cite{MMS} and \cite{Mas}, we define a topological space $X$ to be
\begin{itemize}
\item {\em $\sigma$-metrizable} if $X$ can be written as the union of an increasing sequence $(X_n)_{n\in\w}$ of closed metrizable subspaces of $X$;
\item {\em strongly $\sigma$-metrizable} if $X$ can be written as the union of an increasing sequence $(X_n)_{n\in\w}$ of closed metrizable subspaces of $X$ such that every convergent sequence in $X$ is contained in some space $X_n$;
\item {\em super $\sigma$-metrizable} if $X$ can be written as the union of an increasing sequence $(X_n)_{n\in\w}$ of closed metrizable subspaces of $X$ such that every compact subset of $X$ is contained in some space $X_n$.
\end{itemize}

It turns out that the last two notions are equivalent (which answers a question of V.K.~Maslyuchenko and O.I.~Filipchuk). In the proof we shall use a recent result of Alas and Wilson \cite{AW} (see also \cite{BN}) on sequentially compact spaces. We recall that a topological space $X$ is {\em sequentially compact} if each sequence in $X$ has a convergent subsequence.

\begin{theorem}[Alas, Wilson]\label{AW} Each hereditarily Lindel\"of compact space is sequentially compact.
\end{theorem}

This theorem implies

\begin{corollary}\label{c:seq} Each compact $\sigma$-metrizable topological space $X$ is hereditarily Lindel\"of and sequentially compact.
\end{corollary}

\begin{proof}  The space $X$, being $\sigma$-metrizable, admits a countable cover $\{X_n\}_{n\in\w}$ by closed metrizable subspaces. Each metrizable subspace $X_n$, being closed in the compact space $X$, is compact and hence hereditarily Lindel\"of. Then the union $X=\bigcup_{n\in\w}X_n$ is hereditarily Linde\"of, too. By Theorem~\ref{AW}, the compact space $X$ is sequentially compact.
\end{proof}

For Hausdorff spaces the following characterization was proved in the PhD Thesis \cite{Fil} of Filipchuk.

\begin{theorem}\label{t1} A topological space is strongly $\sigma$-metrizable if and only if it is super $\sigma$-metrizable.
\end{theorem}

\begin{proof} The ``only if'' part trivially follows from the observation that for any sequence $\{x_n\}_{n\in\w}\subset X$, convergent to a point $x\in X$, the subspace $K=\{x\}\cup\{x_n\}_{n\in\w}$ of $X$ is compact.

To prove the ``if'' part, assume that a topological space $X$  is strongly $\sigma$-metrizable. Then $X$ is the union of an increasing sequence $(X_n)_{n\in\w}$ of closed metrizable subspaces of $X$ such that every convergent sequence in $X$ is contained in some set $X_n$. To prove that $X$ is super $\sigma$-metrizable, we need to show that any compact subset $K\subset X$ is contained in some $X_n$. To derive a contradiction, assume that for every $n\in\w$ the complement $K\setminus X_n$ is not empty and hence contains some point $x_n$. By Corollary~\ref{c:seq}, the compact $\sigma$-metrizable space $K$ is sequentially compact. Consequently, there exists an increasing number sequence $(n_i)_{i\in\w}$ such that the subsequence $(x_{n_i})_{i\in\w}$ of $(x_n)_{n\in\w}$ converges in $K$. The choice of the sequence $(X_n)_{n\in\w}$ guarantees that $\{x_{n_i}\}_{i\in\w}\subset X_m$ for some $m\in\w$. Choose a number $i\in\w$ with $n_i\ge m$ and observe that the inclusion $x_{n_i}\in X_m$ contradicts the choice of $x_{n_i}\in X\setminus X_{n_i}\subset X\setminus X_m$.
This contradiction shows that $K\subset X_n$ for some $n\in\w$, which means that the space $X$ is super $\sigma$-metrizable.
\end{proof}

Theorem~\ref{t1} and \cite[3.1.19]{En} implies an interesting metrization theorem for compact topological spaces.

\begin{corollary}\label{c1} For a compact space $X$ the following conditions are equivalent:
\begin{enumerate}
\item $X$ is metrizable;
\item $X$ is Hausdorff and $\sigma$-metrizable;
\item $X$ is strongly $\sigma$-metrizable;
\item $X$ is super $\sigma$-metrizable.
\end{enumerate}
\end{corollary}

The Hausdorff requirement in Corollary~\ref{c1}(2) is essential as shown by the following example.

\begin{example} Let $X$ be a countable infinite space endowed with the Zariski topology $$\tau=\{\emptyset\}\cup\{X\setminus F:\mbox{$F$ is finite}\}.$$
The space $X$ is compact and $\sigma$-metrizable but Hausdorff and not metrizable.
\end{example}

\end{document}